\documentclass[12pt, english]{article}

\usepackage{babel}
\usepackage[latin1]{inputenc}
\usepackage{amsmath}
\usepackage{amssymb}
\usepackage{amsthm}

\theoremstyle{plain}
\newtheorem{Thm}{Theorem}[section]

\newtheorem{Lemma}[Thm]{Lemma}

\theoremstyle{definition}

\newtheorem{Def*}{Definition}

\begin{document}

\title{Rigorous computer analysis of the Chow-Robbins game}

\author{Olle H\"aggstr\"om and Johan W\"astlund \\
\small Department of Mathematical Sciences,\\[-0.8ex]
\small Chalmers University of Technology, \\[-0.8ex] 
\small SE-412 96 Gothenburg, Sweden\\[-0.8ex]
\small \texttt{olleh@chalmers.se, wastlund@chalmers.se}
}
\date{\small \today} 

\maketitle

\begin{abstract} Flip a coin repeatedly, and stop whenever you want. Your payoff is the proportion of heads, and you wish to maximize this payoff in expectation. This so-called Chow-Robbins game is amenable to computer analysis, but while simple-minded number crunching can show that it is best to continue in a given position, establishing rigorously that stopping is optimal seems at first sight to require ``backward induction from infinity''.   

We establish a simple upper bound on the expected payoff in a given position, allowing efficient and rigorous computer analysis of positions early in the game. In particular we confirm that with 5 heads and 3 tails, stopping is optimal.
\end{abstract}


\section{The Chow-Robbins game}
The following game was introduced by Yuan-Shih~Chow and Herbert Robbins \cite{CR} in 1964: We toss a coin repeatedly, and stop whenever we want. Our payoff is the proportion of heads up to that point, and we assume that we want to maximize the expected payoff.

Basic properties of this game, like the fact that there is an optimal strategy that stops with probability 1, were established in \cite{CR}.
Precise asymptotical results were obtained by Aryeh Dvoretzky \cite{D} and Larry Shepp \cite{S}. 
In particular Shepp showed that for the optimal strategy, the proportion of heads required for stopping after $n$ coin tosses is asymptotically $$\frac12 + \frac{0.41996\dots}{\sqrt{n}},$$ where the constant is the root of a certain integral equation.
But as was pointed out more recently by Luis Medina and Doron Zeilberger \cite{MZ}, for a number of positions early in the game the optimal decisions were still not known rigorously. 

Let $V(a,n)$ be the expected payoff under optimal play from position $(a, n)$, by which we mean $a$ heads out of $n$ coin flips. The game is suitable for computer analysis, but there is a fundamental problem in that it seems one has to do ``backward induction from infinity'' in order to determine $V(a,n)$. Clearly \begin{equation} \label{rec} V(a,n) = \max\left(\frac an, \frac{V(a, n+1) + V(a+1,n+1)}2\right),\end{equation} but the ``base case'' is at infinity.

\section{Lower bound on $V(a,n)$}
In position $(a,n)$ we can guarantee payoff $a/n$ by stopping. Moreover, if $a/n<1/2$, then by the recurrence of simple random walk on $\mathbb{Z}$, we can wait until the proportion of heads is at least $1/2$. Therefore \begin{equation} \label{lower} V(a,n) \geq \max\left(\frac{a}{n},\frac12\right).\end{equation}
We can recursively establish better lower bounds by starting from the inequality \eqref{lower} at a given ``horizon'', and then working our way backwards using \eqref{rec}. An obvious approach is letting the horizon consist of all positions with $n=N$ for some fixed $N$. In practice it is more efficient to use \eqref{rec} only for positions where in addition $a\approx n/2$, say when $\left| a - n/2\right| \leq c\sqrt{N}$ for some suitable constant $c$, and to resort to \eqref{lower} outside that range. This allows a greater value of $N$ at given computational resources.

If in this way we find that $V(a,n) > a/n$, then in position $(a,n)$, continuing is better than stopping. For instance it is straightforward to check (see the discussion in \cite{MZ}) that $V(2,3)>2/3$, from which it follows that with 2 heads versus 1 tails, we should continue.

The third column of Table~\ref{theTable} (in the Appendix) shows positions for which we have determined that continuing is better than stopping. These results are based on a calculation with a horizon stretching out to $n=10^7$. They agree with \cite[Section 5]{MZ} with one exception: Medina and Zeilberger conjecture based on calculations with a horizon of 50000 that, in the notation of \cite{D, MZ, S}, $\beta_{127} = 9$, meaning that the difference (number of heads minus number of tails) required in order to stop after 127 flips is 9. Accordingly they suggest stopping with 68--59, but our computation shows that continuing is slightly better.

On the other hand, in order to conclude that stopping is ever optimal, we need a nontrivial upper bound on $V(a,n)$. Clearly such an upper bound cannot come from \eqref{rec} alone, since that equation is satisfied by $V(a,n)\equiv 1$.

\section{Upper bound on $V(a,n)$}
We let $\tilde{V}(a,n)$ be the expected payoff from position $(a,n)$ under \emph{infinite clairvoyance}, that is, assuming we have complete knowledge of the results of the future coin flips and stop when we reach the maximum proportion of heads. Obviously $V(a,n) \leq \tilde{V}(a,n)$, so that any upper bound on $\tilde{V}(a,n)$ is also an upper bound on $V(a,n)$.

\begin{Thm} \label{T:GrandUnifiedInequality}
\begin{equation} \label{GrandUnifiedInequality}  \tilde{V}(a,n) \leq \max\left(\frac{a}{n}, \frac12\right) + \min\left(\frac14\sqrt{\frac{\pi}{n}},\, \frac1{2\cdot\left|2a-n\right|}\right).\end{equation}
\end{Thm}

The first term of the right hand-side of \eqref{GrandUnifiedInequality} is equal to the lower bound \eqref{lower}, and thus the second term bounds the error in that approximation. The proof of Theorem~\ref{T:GrandUnifiedInequality} consists of Lemma~\ref{L:ollestrick} together with some calculations in the rest of Section~\ref{S:proof}.

Let us already here describe how we have used \eqref{GrandUnifiedInequality} computationally. We have computed upper bounds on $V(a,n)$ in a box stretching out to $n \leq N = 10^7$, and with height given by $\left|2a - n\right| \leq h$ for a fixed $h$ (thus the box includes points where $a$ deviates from $n/2$ by at most $h/2$). At the positions on the ``boundary'' of the box (more precisely, where $(a+1,n+1)$ or $(a,n+1)$ is outside the box), $V(a,n)$ has been estimated by \eqref{GrandUnifiedInequality}, whereas for the positions in the interior we have used \eqref{rec}, controlling the arithmetic so that all roundings go up, in order to achieve rigorous upper bounds.

The second term of the right hand-side of \eqref{GrandUnifiedInequality} gives two different upper bounds on the error in \eqref{lower}, where the bound $(1/4)\cdot\sqrt{\pi/n}$ is better close to the line $a=n/2$, while $1/(2\left|2a-n\right|)$ is the sharper one away from that line.
It seemed natural to choose the height $h$ of the box in such a way that these two bounds approximately coincide at the farther corners of the box, in other words so that $$\frac14\sqrt{\frac{\pi}{N}} \approx \frac1{2\cdot h},$$ that is, $h \approx (2/\sqrt{\pi})\cdot \sqrt{N}$. In our computations leading to the results of Table~\ref{theTable} (with $N=10^7$), we have taken $h=3568$. The second column of Table~\ref{theTable} lists positions for which we have determined that stopping is optimal. This includes 5 heads to 3 tails, a position discussed in \cite{MZ} and for which computational evidence \cite{MZ, W} strongly suggested that stopping should be optimal. To the best of our knowledge our computation provides the first rigorous verification of this fact.

\section{Proof of Theorem~\ref{T:GrandUnifiedInequality}} \label{S:proof}
For $a$ and $n$ as before, and $p\in [0,1]$, let $P(a,n,p)$ denote the probability that, starting from position $(a,n)$, at some point now or in the future the total proportion of heads will strictly exceed $p$. In other words $P(a,n,p)$ is the probability of success starting from $(a,n)$ if instead of trying to maximize expected payoff, we try to achieve a proportion of heads exceeding $p$, and continue as long as this has not been achieved. When $p$ is rational, $P(a,n,p)$ is algebraic and can in principle be calculated with the method of \cite{Stadje}, but we need an inequality that can be analyzed averaging over $p$.  

\begin{Lemma} \label{L:ollestrick}
Suppose that in position $(a,n)$, the nonnegative integer $k$ is such that at least $k$ more coin flips will be required in order to obtain a proportion of heads exceeding $p$. Then \begin{equation} \label{ineq} P(a,n,p) \leq \frac1{(2p)^k}.\end{equation}
\end{Lemma}

\begin{proof}
We can assume that $p>\max(a/n, 1/2)$, since otherwise the statement is trivial. From position $(a,n)$ condition on the event that the total proportion of heads will at some later point exceed $p$. Then, by the law of large numbers, there must be a \emph{maximal} $m$ such that after a total of $m$ coin flips the proportion of heads exceeds $p$. Conditioning further on $m$, the number of heads in coin flips number $n+1,\dots,m$ is determined, and all permutations of the outcomes of these $m-n$ coin flips are equally likely. The proportion of heads among these coin flips is at least $p$, so the (conditional) probability that coin flip $n+1$ results in heads is at least $p$. If $k>1$, then if coin flip $n+1$ was heads, the proportion of heads in flips $n+2,\dots,m$ is still at least $p$, so the probability of heads-heads in flips $n+1$ and $n+2$ is at least $p^2$ etc. Therefore the (conditional) probability that flips $n+1,\dots,n+k$ all result in heads is at least $p^k$, and since this holds for every $m$, we don't have to condition on a specific $m$, but only on the event that the proportion of heads will exceed $p$ at some point. 

Since the unconditional probability of $k$ consecutive heads is $1/2^k$, the statement now follows from a simple calculation: On one hand,
$$Pr(\text{$k$ consecutive heads } | \text{ proportion $p$ is eventually exceeded}) \geq p^k.$$
On the other hand,
\begin{multline} \notag Pr(\text{$k$ consecutive heads } | \text{ proportion $p$ is eventually exceeded}) \\ \leq 
\frac{Pr(\text{$k$ consecutive heads})}{Pr(\text{proportion $p$ eventually exceeded})} = \frac{(1/2)^k}{P(a,n,p)}.\end{multline}
Rearranging, we obtain \eqref{ineq}.
\end{proof}

Our next task is to use Lemma~\ref{L:ollestrick} to estimate $\tilde{V}(a,n)$. We have \begin{equation} \label{integration} \tilde{V}(a,n) = \int_0^1 P(a,n,p)\,dp = \max\left(\frac an, \frac12\right) + \int_{\max\left(\frac an, \frac12\right)}^1 P(a,n,p)\,dp.\end{equation}
If $p>\max(a/n, 1/2)$, then the requirement that at least $k$ more coin flips are needed to obtain a proportion of heads exceeding $p$ is equivalent to $$\frac{a+k-1}{n+k-1} \leq p,$$ which we rearrange as $$k\leq 1 + \frac{np-a}{1-p}.$$ 
Since there is always an integer $k$ in the interval $$\frac{np-a}{1-p} \leq k \leq 1 +  \frac{np-a}{1-p},$$ we conclude using Lemma~\ref{L:ollestrick} that for $p$ in the range $\max(a/n,1/2) < p < 1$ of integration in \eqref{integration}, $$P(a,n,p) \leq \frac{1}{(2p)^{\frac{np-a}{1-p}}}.$$
It follows that $$\tilde{V}(a,n) \leq \max\left(\frac{a}{n},\frac12\right) + \int_{\max\left(\frac{a}{n},\frac12\right)}^1 \frac{dp}{(2p)^{\frac{np-a}{1-p}}}.$$
By the substitution $2p=1+t$ and the elementary inequality $$\frac{\log(1+t)}{1-t} \geq t,$$ we obtain 
\begin{multline}\tilde{V}(a,n) \leq \max\left(\frac{a}{n},\frac12\right) + \frac12\int_{\max\left(\frac{2a-n}{n},0\right)}^1 \frac{dt}{(1+t)^{\frac{(1+t)n-2a}{1-t}}} 
\\= \max\left(\frac{a}{n},\frac12\right) + \frac12\int_{\max\left(\frac{2a-n}{n},0\right)}^1 \exp\left(-\frac{(1+t)n-2a}{1-t}\cdot \log(1+t)\right)\,dt
\\ \leq \max\left(\frac{a}{n},\frac12\right) + \frac12\int_{\max\left(\frac{2a-n}{n},0\right)}^1 \exp\left(-(1+t)tn+2at\right)\,dt.
\end{multline} 
By putting $u=t\sqrt{n}$ and replacing the upper bound of integration by infinity, we arrive at 
\begin{equation} \label{maxcases} \tilde{V}(a,n) \leq \max\left(\frac{a}{n},\frac12\right) + \frac1{2\sqrt{n}}\int_{\max\left(\frac{2a-n}{\sqrt{n}},0\right)}^\infty \exp\left(-u^2 + \frac{2a-n}{\sqrt{n}}\cdot u\right)\,du.\end{equation}
Now notice that by the substitution $w= u - (2a-n)/\sqrt{n}$, 
\begin{equation}
\int_{\frac{2a-n}{\sqrt{n}}}^\infty \exp\left(-u^2+\frac{2a-n}{\sqrt{n}}\cdot u\right)\, du = \int_0^\infty \exp\left(-w^2-\frac{2a-n}{\sqrt{n}}\cdot w\right)\, dw.\end{equation}
Therefore regardless of the sign of $2a-n$, \eqref{maxcases} can be written as
\begin{equation}\label{uniform} \tilde{V}(a,n) \leq \max\left(\frac{a}{n},\frac12\right) + \frac1{2\sqrt{n}}\int_0^\infty \exp\left(-u^2 - \frac{\left|2a-n\right|}{\sqrt{n}}\cdot u\right)\,du.\end{equation}

The bound \eqref{uniform} can be used directly in computations by first tabulating values of the integral, but we have chosen to simplify the error term further (instead spending computer resources on pushing the horizon).
We can discard either of the two terms inside the exponential in \eqref{uniform}. On one hand, the error term is at most  
$$\frac1{2\sqrt{n}}\int_0^\infty \exp\left(-u^2\right)\,du = \frac14\sqrt{\frac{\pi}{n}}.$$
On the other hand, it is also bounded by 
$$ \frac1{2\sqrt{n}}\int_0^\infty\exp\left(-\frac{\left|2a-n\right|}{\sqrt{n}}\cdot u\right)\,du =\frac1{2\cdot\left|2a-n\right|}.$$
This completes the proof of Theorem~\ref{T:GrandUnifiedInequality}. 

In the latter case, $\left|2a-n\right|$ is the absolute difference between the number of heads and the number of tails. The simplicity of the inequality $\tilde{V}(a,n) \leq \max(a/n,1/2) + 1/(2\left|2a-n\right|)$ suggests that there might be a proof involving considerably less calculation. 

Theorem~\ref{T:GrandUnifiedInequality} allows us to calculate $V(a,n)$ to any desired precision. This is because \eqref{rec} has the property that if $V(a,n+1)$ and $V(a+1,n+1)$ are both known with an error of at most $\varepsilon$, then the same is true of $V(a,n)$. To obtain the desired level of precision, we therefore only need to start our calculation from a horizon where the error term in \eqref{GrandUnifiedInequality} is sufficiently small.

On the other hand it is difficult to say in advance how far we have to take our computations in order to find \emph{the optimal decision} in a given position, as the expected payoff on continuing may be very close to the payoff $a/n$ on stopping. For instance, we have no idea how hard it will be to find the optimal decision in the position 116--104 (the first one whose status we haven't determined). For all we know the question whether stopping is optimal in this position might be undecidable by our method, although this would require the expected payoff on continuing to miraculously be exactly equal to the payoff on stopping.

\newpage

\section{Appendix: Computational results}
We have computed upper and lower bounds on $V(a,n)$ for $(a,n)$ satisfying $n\leq 10^7$ and $\left|a-n/2\right| \leq 1784 \approx \sqrt{10^7/\pi}$. These results allow us to find the optimal decision in most positions early in the game. It is better to continue precisely when $V(a,n) > a/n$, while stopping is optimal when $V(a,n) = a/n$.

We have included the results relevant to a total of at most 1000 coin flips, and in this range we have determined optimal play for all except seven positions.

If the number $a$ of heads is not greater than $n/2$, continuing is always better than stopping. If $a > n/2$, then to read Table~\ref{theTable}, consider the difference $2a-n = a - (n-a)$ of the number of heads to the number of tails. It turns out (as is easily shown by a coupling argument) that for a fixed difference, the optimal decision will be to stop if $n$ is below a certain threshold, and to continue if $n$ is above that threshold. 

If for instance we have 19 heads against 14 tails, the difference is 5. According to the table, stopping is best even up to 23--18, so we stop. As can be seen in the table, the opening theory is complete up to difference 11, while for difference 12 the status of the position 116--104 is still unknown.

For the position 16--12, the decision is extremely close, and a run with $N=10^6$ fails to determine the optimal decision, giving an upper bound of $0.5714326$ on continuing compared to the payoff $16/28 \approx 0.57142857$ on stopping. A run with $N=10^7$ shows that the expected payoff on continuing is between 0.5714192 and 0.5714278, revealing that stopping is optimal. 

For $V(0,0)$, Julian Wiseman gives the lower bound $0.7929534812$ based on a calculation \cite{W} much more extensive than ours (with a horizon of $N=2^{28} \approx 268,000,000$) and suggests $0.79295350640$ as an approximation of the true value. Our bounds obtained with $N=10^7$ are $$0.79295301268091 < V(0,0) < 0.79295559864361.$$

\newpage
\begin{table} [h]
\begin{center}
\begin{tabular} {ccc}
difference & stop with & but go with\\
1 & 1--0 & 2--1 \\
2 & 5--3 & 6--4\\
3 & 9--6 & 10--7\\
4 & 16--12 & 17--13\\
5 & 23--18 & 24--19\\
6 & 32--26 & 33--27\\
7 & 42--35 & 43--36\\
8 & 54--46 & 55--47\\
9 & 67--58 & 68--59\\
10 & 82--72 & 83--73\\
11 & 98--87 & 99--88\\
12 & 115--103 & 117--105\\
13 & 134--121 & 135--122\\
14 & 155--141 & 156--142\\
15 & 176--161 & 177--162\\
16 & 199--183 & 201--185\\
17 &  224--207 & 225--208 \\
18 & 250--232 & 251--233\\
19 & 277--258 & 279--260\\
20 & 306--286 & 307--287\\
21 & 336--315 & 338--317\\
22 & 368--346 & 369--347 \\
23 & 401--378 & 402--379\\
24 & 435--411 & 437--413\\
25 & 471--446 & 473--448\\
26 & 508--482 & 510--484\\
$\geq 27$ & stop & 

\end{tabular}
\end{center}
\caption{Opening theory for the first 1000 steps of the Chow-Robbins game. If the difference (number of heads $-$ number of tails) is non-positive, we always continue. If the difference is 27 or more and the total number of flips is at most 1000, stopping is optimal. For differences from 1 to 26, stopping is optimal up to and including the position in column 2, while continuing is optimal from the position in column 3 and on. There are seven positions in this range for which we have not determined the optimal decision: 116--104, 200--184, 278--259, 337--316, 436--412, 472--447 and 509--483.}
\label{theTable}
\end{table}

\end{document}